\documentclass[12pt]{amsart}
\usepackage[utf8]{inputenc}
\usepackage[francais]{babel}
\usepackage{pifont}
\usepackage{wasysym}
\usepackage{bbm}
\usepackage{extarrows}
\usepackage{amssymb,xspace} 
\usepackage{amstext}
\usepackage{pdfsync}
\usepackage{hyperref} 
\usepackage[mathscr]{eucal} 
\theoremstyle{plain}
\usepackage{amsbsy,amssymb,amsfonts,latexsym,eucal,amscd}
\usepackage[all]{xy}

\newcommand{\tens}[1][]{\mathbin{\otimes_{\raise1.5ex\hbox to-.1em{}{#1}}}}

\newtheorem{theorem}{Theorème}[section]

\newtheorem{corollary}[theorem]{Corollaire}

{\theoremstyle{definition}}

{\theoremstyle{definition}}

{\theoremstyle{definition}}

{\theoremstyle{definition}}

{\theoremstyle{definition}\newtheorem{definition}[theorem]{Definition}}

{\theoremstyle{definition}}

{\theoremstyle{definition}}

\author{Julien Grivaux}

\address{Sorbonne Université \\ 
Institut de Mathématiques de Jussieu-Paris Rive Gauche \\
Case 247, 4 place Jussieu \\
F-75005, Paris, France.}

\email{jgrivaux@math.cnrs.fr}

\urladdr{http://jgrivaux.perso.math.cnrs.fr/}

\title{Modèles géométriques attachés aux paires réductives}

\begin{document}

\begin{abstract}
L'objet de cet article d'exposition est de présenter une introduction à l'article \cite{calaque_ext_2017}, et d'expliquer plus généralement comment le dictionnaire entre théorie de Lie et géométrie algébrique, développé dans \cite{calaque_pbw_2013} à la suite des travaux \cite{kapranov_rozansky-witten_1999} et \cite{markarian_atiyah_2009}, peut être consolidé au sens où les énoncés géométriques ne sont plus seulement des analogues de résultats algébriques, mais en sont des extensions pour certains objets de Lie dans des catégories dérivées.
\end{abstract}
\vspace*{1.cm}
\maketitle
\tableofcontents

\section{Introduction}
Les travaux de Kapranov \cite{kapranov_rozansky-witten_1999} et Markarian \cite{markarian_atiyah_2009} ont permis d'associer à toute variété algébrique lisse $X$ sur un corps de caractéristique zéro un crochet de Lie sur le fibré tangent décalé $\mathrm{T}_X[-1]$. Ce crochet de Lie (catégorique) est donné par la classe d'Atiyah du fibré tangent $\mathrm{T}_X$, qui est un élément de $\mathrm{Ext}^1_{\mathcal{O}_X}(\mathrm{T}_X^{\otimes 2}, \mathrm{T}_X)$, ou de manière équivalente un morphisme 
\[
\mathrm{T}_X[-1] \otimes \mathrm{T}_X[-1] \rightarrow \mathrm{T}_X[-1]
\] 
dans $\mathrm{D}(X)$. Tout objet $\mathcal{F}$ de la catégorie dérivée $\mathrm{D}(X)$ devient naturellement une représentation de $\mathrm{T}_X[-1]$ grâce à la classe d'Atiyah de $\mathcal{F}$, qui est un morphisme 
\[
\mathrm{T}_X[-1] \otimes \mathcal{F} \rightarrow \mathcal{F}
\] 
dans $\mathrm{D}(X)$. En se fondant sur ces travaux, Calaque, C\u{a}ld\u{a}raru et Tu ont proposé dans \cite{calaque_pbw_2013} un dictionnaire entre algèbres de Lie et variétés algébriques: les catégories dérivées jouent le rôle des représentations, la cohérence correspond à la finitude, et les opérations de fonctorialité classique (restriction, induction) se transposent dans le cadre géométrique. 
\par \medskip
Dans \cite{calaque_ext_2017}, Calaque et l'auteur ont développé de manière systématique un certain nombre de résultats classiques pour les algèbres de Lie sur un corps de caractéristique zéro, comme l'isomorphisme PBW, ou la formule permettant d'exprimer explicitement la multiplication par des éléments de degré un dans l'algèbre enveloppante, dans le cas d'objets de Lie d'une catégorie monoidale symétrique $\mathbf{k}$-linéaire. Ces résultats permettent de démontrer facilement que l'algèbre enveloppante de $\mathrm{T}_X[-1]$ est isomorphe à la cohomologie de Hochschild de $\mathcal{O}_X$.  L'étape suivante est de comprendre l'analogue géométrique des paires réductives, qui sont les cycles analytiques quantifiés introduits par l'auteur dans \cite{grivaux_hochschild-kostant-rosenberg_2014}; ce sont des paires $X \subset Y$ munies d'une rétraction $\sigma$ à l'ordre $1$. Dans \cite{yu_todd_2015}, Yu introduit une condition supplémentaire sur les cycles quantifiés, que nous appellerons \textit{modération géométrique}: on impose que $\sigma^* \mathrm{N}_{X/Y}$ s'étend en un faisceau localement libre sur le second voisinage formel de $X$ dans $Y$. Dans \cite{calaque_ext_2017}, nous interprétons la condition de Yu comme la géométrisation d'une condition algébrique très naturelle sur les paires réductives. Enfin sous cette hypothèse de modération géométrique, nous pouvons décrire l'algèbre $\mathrm{Ext}^*_{\mathcal{O}_Y}(\mathcal{O}_X, \mathcal{O}_X)$.

\section{Un exemple élémentaire de géométrisation}
Dans cette section, nous présentons en guise d'introduction (avant d'aborder des cadres plus compliqués) une preuve du théorème de d'Alembert-Gau{\ss} par géométrisation.
\par \medskip
Soit $P$ un polynôme à coefficients complexes de degré $n>0$. On peut supposer que $P(0)$ est non nul, car sinon $0$ est une racine de $P$. Dans ce cas l'action de $t$ sur $V:=\mathbb{C}[t] / P \mathbb{C}[t]$ est un endomorphisme $u$ inversible de polynôme caractéristique $P$. La recherche d'une racine de $P$ est ainsi équivalente à l'existence d'un vecteur propre de $u$, et comme $u$ est inversible ceci se réduit à l'existence d'un point fixe de l'action de $u$ sur l'espace projectif $\mathbb{P}(V)$.  
\par \medskip
Le problème initial a donc été reformulé dans la version géométrique suivante: les racines de $P$ correspondent aux point d'intersection de deux sous-variétés de $\mathbb{P}(V) \times \mathbb{P}(V)$: la diagonale et le graphe de $u$. Le calcul du nombre de points d'intersection (avec multiplicités) de deux sous-variétés compactes d'une variété compacte orientée est un calcul homologique: on cherche à calculer
l'image de l'élément $[\Delta_{\mathbb{P}(V)}] \otimes [\Gamma_u]$ par le produit d'intersection
\[
\mathrm{H}_{2n-2}(\mathbb{P}(V) \times \mathbb{P}(V), \mathbb{Z})^{\otimes 2} \xrightarrow{\cap} \mathrm{H}_{0}(\mathbb{P}(V) \times \mathbb{P}(V), \mathbb{Z}) \simeq \mathbb{Z}.
\]
\par \medskip
Dans le cas précis de l'intersection de la diagonale avec un graphe, ce calcul est bien connu: il s'agit de la formule du point fixe de Lefschetz: le nombre d'intersection que l'on cherche est égal au nombre de Lefschetz
\[
L(u)=\sum_{i=0}^{2n-2} (-1)^i \, \mathrm{Tr} \left\{u_*  \circlearrowright \mathrm{H}_{i}(\mathbb{P}(V), \mathbb{Q})  \right\}
\] 
qui est égal à la somme alternée des traces des actions de $u$ en homologie singulière. Dans le cas de l'espace projectif complexe, l'homologie singulière est concentrée en degré pair, et pour tout $i$ entre $0$ et $n-1$, $\mathrm{H}_{2i}(\mathbb{P}(V), \mathbb{Q})$ est de dimension $1$, un générateur canonique étant la classe d'homologie de n'importe quel sous-espace complexe de $V$ de dimension complexe $i$ (ces sous-espaces sont tous homologues). Il s'ensuit que $u_*=\mathrm{id}$ sur chaque droite rationnelle $\mathrm{H}_{2i}(\mathbb{P}(V), \mathbb{Q})$, et donc le nombre de Lefschetz $L(u)$ égale $n$. En particulier, le graphe de $u$ doit intersecter la diagonale, donc $u$ a un point fixe.
\par \medskip
Cette preuve est ainsi de nature complètement \textit{globale}, car elle s'appuie sur la description des groupes d'homologie des espaces projectifs complexes. Le terme de \og{} géométrisation \fg s'explique par le fait que l'on est parti d'un objet de nature plutôt algébrique (un endomorphisme inversible d'un espace vectoriel complexe), et que l'on a ensuite étudié un objet géométrique qui lui était associé, à savoir ici la classe d'homologie du graphe de l'endomorphisme projectivisé. Dans la suite, nous allons nous intéresser à des variantes géométriques de problèmes faisant intervenir des algèbres de Lie de dimension finie sur un corps de caractéristique zéro. La géométrisation sera plus imparfaite au sens où l'on considèrera des problèmes analogues (et non directement construits à partir du problème algébrique initial).

\section{L'isomorphisme de Duflo}

Soit $\mathfrak{g}$ une algèbre de Lie de dimension finie sur un corps $\mathbf{k}$ de caractéristique zéro. Notons $\mathrm{S}(\mathfrak{g})$ l'algèbre symétrique de $\mathfrak{g}$, $\mathrm{T}(\mathfrak{g})$ son algèbre tensorielle, et $\mathrm{U}(\mathfrak{g})$ son algèbre enveloppante. 
Rappelons que $\mathrm{U}(\mathfrak{g})$ est le quotient de  $\mathrm{T}(\mathfrak{g})$ par l'idéal bilatère engendré par les éléments du type $x \otimes y - y \otimes x - [x, y]$ pour $x, y \in \mathfrak{g}$. On peut réaliser $\mathrm{S}(\mathfrak{g})$ comme le sous-espace vectoriel de $\mathrm{T}(\mathfrak{g})$ formé des tenseurs totalement symétriques.
On dispose alors de l'isomorphisme de Poincaré-Birkhoff-Witt: 
\[
\mathrm{I}_{\mathrm{PBW}} \colon \mathrm{S}(\mathfrak{g}) \hookrightarrow \mathrm{T}(\mathfrak{g}) \twoheadrightarrow \mathrm{U}(\mathfrak{g})
\]
Introduisons l'élément de Duflo $\mathfrak{d}$, qui est une série formelle invariante sur $\mathfrak{g}$, c.à.d. un élément de $\widehat{\mathrm{S}}(\mathfrak{g}^*)^{\mathfrak{g}}$, défini par la formule
\[
\mathfrak{d}(x)=\mathrm{det} \left( \frac{\mathrm{ad}(x)}{1-\exp(-\mathrm{ad}(x))} \right).
\]
Cet élément s'exprime comme une série universelle en les polynômes invariants $P_k$ définis par $P_k(x)=\mathrm{Tr} (\mathrm{ad}(x)^k)$: on a
\[
\mathfrak{d}= 1+\frac{P_1}{2}+ \frac{3P_1^2-P_2}{24}+ \frac{P_1^3-P_1P_2}{48}+ \ldots
\]
Cet élément permet de comprendre explicitement le centre de l'algèbre enveloppante $\mathrm{U}(\mathfrak{g})$ grâce au résultat suivant: 

\begin{theorem}[\cite{duflo_operateurs_1977}]
Le morphisme de $\mathrm{S}(\mathfrak{g})^{\mathfrak{g}}$ dans $\mathrm{U}(\mathfrak{g})^{\mathfrak{g}}$ donné par la formule $x \rightarrow \mathrm{I}_{\mathrm{PBW}}(\mathfrak{d}^{-1/2} \lrcorner \,x)$ est un isomorphismes d'algèbres.
\end{theorem}
Ici, $\lrcorner$ désigne la contraction. De plus, il est facile de voir que $\mathrm{U}(\mathfrak{g})^{\mathfrak{g}}$ est le centre de $\mathrm{U}(\mathfrak{g})$.

\section{Conjectures de Duflo pour les paires réductives}
On se donne maintenant une paire réductive $(\mathfrak{h}, \mathfrak{g})$ sur $\mathbf{k}$, c'est-à-dire deux algèbres de Lie de dimension finie $\mathfrak{h} \subset \mathfrak{g}$ ainsi qu'une décomposition $\mathfrak{h}$-équivariante $\mathfrak{g}=\mathfrak{h} \oplus \mathfrak{n}$. On s'intéresse alors à l'algèbre 
\[
A(\mathfrak{h}, \mathfrak{g})=\left( \frac{\mathrm{U}(\mathfrak{g})}{\mathrm{U}(\mathfrak{g}) \mathfrak{h}} \right)^{\mathfrak{h}}.
\]
On dispose à nouveau d'un morphisme de type Poincaré-Birkhoff-Witt $\mathfrak{h}$-équivariant donné par
\[
\mathrm{S}(\mathfrak{n}) \hookrightarrow \mathrm{S}(\mathfrak{g}) \hookrightarrow \mathrm{T}(\mathfrak{g}) \twoheadrightarrow \mathrm{U}(\mathfrak{g}) \twoheadrightarrow \frac{\mathrm{U}(\mathfrak{g})}{\mathrm{U}(\mathfrak{g}) \mathfrak{h}} \cdot
\]
Notons que l'hypothèse \og paire réductive \fg{} est ici trop restrictive, la condition nécessaire et suffisante pour avoir un tel isomorphisme a été établie dans \cite{calaque_pbw_2013}. 
\par \medskip
La conjecture de Duflo généralisée\footnote{La version précise de cette conjecture fait intervenir un caractère de $\mathfrak{h}$, nous ne l'avons pas incluse pour ne pas compliquer l'énoncé.} prédit que le centre de $A(\mathfrak{h}, \mathfrak{g})$ est isomorphe au centre de Poisson de $\mathrm{S}(\mathfrak{n})^{\mathfrak{h}}$. Le cas de la paire réductive diagonale $(\mathfrak{g}, \mathfrak{g}\oplus \mathfrak{g})$ correspond au théorème de Duflo précédemment cité. Cette conjecture reste largement ouverte, même dans le cas des paires symétriques, c.à.d. quand $\mathfrak{h}$ est le lieu fixe d'une involution de $\mathfrak{g}$ préservant le crochet de Lie.

\section{Paires réductives modérées}
Cette notion a été introduite dans \cite{calaque_ext_2017}, nous reviendrons à un analogue géométrique dans la dernière section.
\begin{definition}
Une paire réductive $(\mathfrak{g}, \mathfrak{h}, \mathfrak{n})$ est dite modérée si pour tous $n_1, n_2, n_3$ dans $\mathfrak{n}$, 
$[\pi_{\mathfrak{h}}[n_1, n_2], n_3]=0$.
\end{definition}
Remarquons que si $\mathfrak{h}$ est un idéal dans $\mathfrak{g}$ (ce qui signifie que $\mathfrak{g}=\mathfrak{h} \rtimes \mathfrak{n}$), cette condition est automatiquement satisfaite car $[\mathfrak{h}, \mathfrak{n}] \subset \mathfrak{h} \cap \mathfrak{n}=\{0\}$. 
\par \medskip
La condition de modération implique que le crochet définit sur $\mathfrak{n}$ par la formule $[n_1, n_2]_{\mathfrak{n}}=\pi_{\mathfrak{n}} ([n_1, n_2])$ est un crochet de Lie. De plus $\mathfrak{h}$ agit par dérivation sur $\mathfrak{n}$, et on a un isomorphisme $\mathfrak{h}$-équivariant
\[
\frac{\mathrm{U}(\mathfrak{g})}{\mathrm{U}(\mathfrak{g}) \mathfrak{h}} \simeq \mathrm{U}(\mathfrak{n}).
\]
En particulier, l'isomorphisme PBW pour $\displaystyle \frac{\mathrm{U}(\mathfrak{g})}{\mathrm{U}(\mathfrak{g}) \mathfrak{h}}$ se réduit à l'isomorphisme PBW pour l'algèbre de Lie $\mathfrak{n}$.
\section{Géométrisation de Kapranov-Markarian}
La notion géométrique naturelle attachée à une algèbre de Lie est celle de groupe de Lie. On va procéder ici en sens inverse: partir d'objets géométriques donnant lieu à des algèbres de Lie un peu plus générales que les algèbres de Lie classiques: il s'agira d'objets de Lie dans des catégories monoïdales symétriques.  
\par \medskip
Si $X$ est une variété algébrique lisse sur $\mathbf{k}$ et $E$ est un fibré vectoriel sur $X$, le fibré $\mathrm{J}^1(E)$ des $1$-jets à valeurs dans $E$ s'insère dans une suite exacte
\[
0 \rightarrow E \rightarrow \mathrm{J}^1 (E) \rightarrow \mathrm{T}_X \otimes E \rightarrow 0.
\]
ce qui fournit une classe d'extension $\mathrm{at}_E$ dans $\mathrm{Ext}^1_{\mathcal{O}_X}(\mathrm{T}_X \otimes E, E)$, appelée la classe d'Atiyah de $E$. Cette classe s'annule si et seulement si $E$ admet une connexion algébrique globale. Dans la suite, on considèrera $\mathrm{at}_E$ comme un morphisme de $\mathrm{T}_X [-1] \otimes E \rightarrow \mathrm{T}_X[-1]$ dans la catégorie dérivée $\mathrm{D}(X)$ des complexes de faisceaux de $\mathcal{O}_X$-modules. On a alors le résultat suivant: 

\begin{theorem}[\cite{kapranov_rozansky-witten_1999}, \cite{markarian_atiyah_2009}]
Étant donné une variété algébrique lisse $X$ sur $\mathbf{k}$, la classe d'Atiyah de $\mathrm{T}_X$, considérée comme un morphisme $\mathrm{T}_X[-1] \otimes \mathrm{T}_X[-1] \rightarrow \mathrm{T}_X[-1]$ dans $\mathrm{D}(X)$, définit une structure d'objet de Lie sur $\mathrm{T}_X[-1]$.
\end{theorem}
On va donc s'intéresser à une classe particulière d'objets de Lie attachés à des variétés algébriques lisses, il s'agit d'objets de Lie du type $\mathrm{T}_X[-1]$ où $X$ est lisse sur $\mathbf{k}$. Bien que ces objets ne soient pas exactement des algèbres de Lie au sens usuel, la géométrisation correspondante est à la fois structurellement très riche, car permettant d'utiliser tous les outils de géométrie algébrique, et relativement simple, car ne faisant pas intervenir de singularités ou de structures champêtres.
\par \medskip
Notons que cette construction est réalisable pour une catégorie d'objets géométriques plus généraux que les variétés algébriques lisses: dans \cite{hennion_tangent_2018} l'auteur montre que si $X$ est un champ d'Artin dérivé, le complexe tangent décalé $\mathbb{T}_X[-1]$ est également muni d'une structure de Lie. Dans le cas où $G$ est un groupe algébrique et $X=\mathrm{B}G$, $\mathbb{T}_X=\mathfrak{g}[1]$ et la structure de Lie sur $\mathbb{T}_X[-1]$ est exactement celle de l'algèbre de Lie $\mathfrak{g}$ de $G$.

\section{L'isomorphisme PBW géométrique}
L'isomorphisme de Poincaré-Birkhoff-Witt pour les algèbres de Lie classiques sur un corps de caractéristique zéro est un résultat élémentaire, mais néanmoins délicat. Parmi les preuves les plus conceptuelles, on trouve celle présentée dans \cite{braverman_poincarebirkhoffwitt_1996} utilisant la dualité de Koszul. L'isomorphisme PBW est en fait valable dans toute catégorie monoïdale symétrique karoubienne  $\mathbf{k}$-linéaire: ce résultat bien connu dans le folklore est montré rigoureusement dans \cite{calaque_ext_2017}, en combinant une stratégie dévelopée dans \cite{deligne_notes_1999} et une méthode opéradique. La difficulté principale dans le cas d'une catégorie non abélienne est de construire une algèbre $\mathcal{A}$ quotient de $\mathrm{T} (\mathfrak{g})$ satisfaisant les deux conditions suivantes: 
\begin{enumerate}
\item[(A1)] La composition
$
\mathrm{S}(\mathfrak{g}) \rightarrow \mathrm{T} (\mathfrak{g}) \rightarrow \mathcal{A}
$
est un isomorphisme. 
\item[(A2)]
Pour tout $p$, le morphisme $\mathrm{T}^p \mathfrak{g}  \rightarrow \mathrm{Gr}^p \mathcal{A} \simeq \mathrm{S}^p \mathfrak{g}$
est la symétrisation.
\end{enumerate}
Il est alors relativement aisé de montrer (il s'agit du \og reverse categorical PBW theorem \fg{} de \cite{calaque_ext_2017}) que $\mathcal{A}$ est l'algèbre enveloppante de $\mathfrak{g}$, au sens ou pour toute algèbre $\mathcal{B}$, 
\[
\mathrm{Hom}_{\mathrm{Lie}}(\mathfrak{g}, \mathrm{Lie}(\mathcal{B})) \simeq \mathrm{Hom}_{\mathrm{Alg}}(\mathrm{U}(\mathfrak{g}), \mathcal{B}).
\]
Ici, $ \mathrm{Lie}(\mathcal{B})$ est l'algèbre $\mathcal{B}$ munie du crochet de Lie donné par le commutateur. 
\par \medskip
Dans le cas géométrique, une telle algèbre est facile à produire, ce qui va nous permettre de démontrer le théorème suivant: 
\begin{theorem}[\cite{calaque_ext_2017}, \cite{markarian_atiyah_2009}, \cite{ramadoss_relative_2008}]
Étant donné une variété algébrique lisse sur $\mathbf{k}$, l'algèbre enveloppante de l'objet de Lie $\mathrm{T}_X[-1]$ est isomorphe à $\mathrm{pr}_{1*} \mathcal{RH}om_{\mathcal{O}_{X \times X}}(\mathcal{O}_{\Delta_X}, \mathcal{O}_{\Delta_X})$.
\end{theorem}

\begin{proof}
On considère l'algèbre $\mathrm{pr}_{1*} \mathcal{RH}om_{\mathcal{O}_{X \times X}}(\mathcal{O}_{\Delta_X}, \mathcal{O}_{\Delta_X})$. Si $\mathcal{I}$ est l'idéal de la diagonale dans $X \times X$, on a une suite exacte
\[
0 \rightarrow \frac{\mathcal{I}}{\mathcal{I}^2} \rightarrow \frac{\mathcal{O}_{X \times X}}{\mathcal{I}^2} \rightarrow \frac{\mathcal{O}_{X \times X}}{\mathcal{I}} \rightarrow 0
\]
et donc un morphisme $\phi$ de $\delta_* \mathcal{O}_{X}$ dans $\delta_* \Omega_X[1]$ dans $\mathrm{D}(X \times X)$, où $\delta$ est l'injection diagonale. On en déduit un morphisme
\begin{align*}
\mathrm{T}_X[-1]& \,\simeq \,  \mathcal{RH}om_{\mathcal{O}_X}(\Omega_X[1], \mathcal{O}_X)\\
 &\rightarrow \mathrm{pr}_{1*} \mathcal{RH}om_{\mathcal{O}_{X \times X}}(\delta_* \Omega_X[1], \delta_* \mathcal{O}_X) \\
&\xrightarrow{ \circ \phi } \mathrm{pr}_{1*} \mathcal{RH}om_{\mathcal{O}_{X \times X}}(\delta_* \mathcal{O}_X, \delta_* \mathcal{O}_X) 
\end{align*}
La propriété universelle de l'algèbre tensorielle fournit alors un morphisme 
\begin{equation} \label{yolo}
\mathrm{T}(\mathrm{T}_X[-1]) \rightarrow \mathrm{pr}_{1*} \mathcal{RH}om_{\mathcal{O}_{X \times X}}(\delta_* \mathcal{O}_X, \delta_* \mathcal{O}_X) 
\end{equation}
et il est facile de vérifier par des calculs locaux que les propriétés (A1) et (A2) sont satisfaites.
\end{proof}
On peut remarquer que la géométrisation permet de réaliser de manière alternative certaines de ces constructions: le morphisme \eqref{yolo} peut être construit comme le morphisme
\[
 \mathrm{pr}_{1*} \mathcal{RH}om_{{\mathcal{O}_{X \times X}}/{\mathcal{I}^2}}(\mathcal{O}_{\Delta_X}, \mathcal{O}_{\Delta_X})  \rightarrow \mathrm{pr}_{1*} \mathcal{RH}om_{\mathcal{O}_{X \times X}}(\mathcal{O}_{\Delta_X}, \mathcal{O}_{\Delta_X}) .
\]
De plus, l'isomorphisme 
\[
\mathrm{S}(\mathrm{T}_X[-1]) \simeq  \mathrm{pr}_{1*} \mathcal{RH}om_{\mathcal{O}_{X \times X}/\mathcal{I}^2}(\mathcal{O}_{\Delta_X}, \mathcal{O}_{\Delta_X})
\] 
fournissant la propriété (A1) est l'isomorphisme de Hochschild-Kostant-Rosenberg géométrique. Pour ces énoncés, on renvoie le lecteur aux articles \cite{arinkin_when_2012} et \cite{grivaux_hochschild-kostant-rosenberg_2014}. Enfin, mentionnons que l'isomorphisme de Duflo dans ce cadre a été démontré dans \cite{calaque_hochschild_2010}, on renvoie à \cite{calaque_lectures_2011} pour une exposition détaillée de ce résultat difficile.

\section{Cycles quantifiés modérés}
La notion de cycle quantifié, introduite dans \cite{grivaux_hochschild-kostant-rosenberg_2014}, est une géométrisation de celle de paire réductive. Le modèle géométrique que nous associons à une paire d'algèbres de Lie $(\mathfrak{h}, \mathfrak{g})$ où $\mathfrak{h}$ est une sous-algèbre de Lie de $\mathfrak{g}$, est une paire $(X, Y)$ de variétés algébriques lisses, où $X$ est une sous-variété fermée de $Y$. L'inclusion d'algèbres de Lie $\mathfrak{h} \hookrightarrow \mathfrak{g}$ est alors donnée dans ce contexte par le morphisme $\mathrm{T}_X[-1] \rightarrow \mathrm{T}_Y[-1]_{|X}$.
\par \medskip
Le cadre géométrique correspondant aux paires réductives est donc la donnée d'un isomorphisme $\mathrm{T}_Y[-1]_{|X} \simeq\mathrm{T}_X[-1] \oplus \mathrm{N}_{X/Y}[-1]$, c'est à dire un scindage global de la suite normale
\[
0 \rightarrow \mathrm{T}X \rightarrow \mathrm{T}Y_{|X} \rightarrow \mathrm{N}_{X/Y} \rightarrow 0.
\] 
Une telle paire $(X, Y)$ munie d'un scindage de la suite normale (qui n'existe pas toujours) est appelé un cycle quantifié. Le scindage peut être encodé par une rétraction  globale de l'application de $X$ vers son premier voisinage formel $X_Y^{(1)}$ dans $Y$.
\par \medskip
On peut maintenant hiérarchiser les paires réductives en trois niveaux: 
\par \medskip
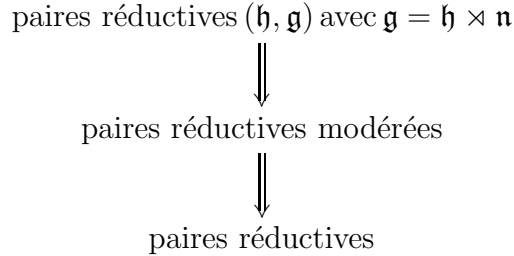
\[
\xymatrix{ \textrm{paires réductives}\, (\mathfrak{h}, \mathfrak{g})\, \textrm{avec}\, \mathfrak{g}=\mathfrak{h} \rtimes \mathfrak{n} \ar@{=>}[d] \\
\textrm{paires réductives modérées} \ar@{=>}[d] \\
\textrm{paires réductives}
}
\]
\par \medskip
Comprendre à quoi correspondent ces conditions dans le cas d'une paire géométrique n'est pas immédiat. Pour le cas où $\mathfrak{h}$ est un idéal, cela correspond à l'existence d'un relèvement de la quantification $\sigma \colon X_Y^{(1)} \rightarrow X$ à $X_Y^{(2)}$.
\[
\xymatrix{ X_Y^{(2)} \ar@{-->}[rd]& \\
 X_Y^{(1)} \ar[u] \ar@<2pt>[r]^-{\sigma} & X \ar@<2pt>[l]
}
\]
En d'autres termes, la rétraction à l'ordre un $\sigma$ doit admettre une extension à l'ordre deux.
\par \medskip
Dans le cas modéré, la condition géométrique a été découverte (indépendamment de toute considération de théorie de Lie) dans l'article \cite{yu_todd_2015}: nous dirons que la paire $(X, Y)$ est modérée si le faisceau localement libre $\sigma^* \mathrm{N}_{X/Y}$ sur $X^{(1)}_Y$ s'étend en un faisceau localement libre sur le second voisinage formel $X^{(2)}_Y$.
\par \medskip
Dans le cas d'une paire géométrique modérée $(X, Y)$, le fibré normal décalé $\mathrm{N}_{X/Y}[-1]$ est donc naturellement muni d'une structure de Lie. Le résultat principal que nous établissons est le suivant: 

\begin{theorem}[\cite{calaque_ext_2017}]
Soit $(X, \sigma)$ un cycle quantifié modéré dans $Y$, et $\mathcal{J}$ l'idéal de $X$ dans $Y$. Alors: 
\begin{enumerate}
\item[--] La structure de Lie sur $\mathrm{N}_{X/Y}[-1]$ découlant de l'hypothèse de modération est associée à la classe d'extension de la suite exacte 
\[
0 \rightarrow \frac{\mathcal{J}^2}{\mathcal{J}^3} \rightarrow \sigma_* \frac{\mathcal{J}}{\mathcal{J}^3} \rightarrow \frac{\mathcal{J}}{\mathcal{J}^2} \rightarrow 0.
\]
\item[--] Les objets $\mathcal{RH}om_{O_Y}^{\ell}(\mathcal{O}_X, \mathcal{O}_X)$ et $\mathcal{RH}om_{O_Y}^{r}(\mathcal{O}_X, \mathcal{O}_X)$ de $\mathrm{D}(X)$ sont naturellement deux algèbres dans $\mathrm{D}(X)$, et toutes les deux sont isomorphes à l'algèbre enveloppante $\mathrm{U}(\mathrm{N}_{X/Y}[-1])$.
\item[--] L'élément de Duflo de $\mathrm{N}_{X/Y}[-1]$ est égal à la classe de cycle quantifiée $\{X\}_{\sigma}$ introduite dans \cite{grivaux_hochschild-kostant-rosenberg_2014}.
\end{enumerate}
\end{theorem}
Quelques commentaires sur ce résultat: le premier point est une généralisation de la construction de la structure de Lie de Kapranov-Markarian. Le second point est le plus délicat: si on dérive $\mathcal{H}om$ comme bi-foncteur, on obtient une algèbre dans $\mathrm{D(Y)}$. Les exposants \og r \fg{} et \og $\ell$ \fg{} signifient que le foncteur $\mathcal{H}om$ n'est dérivé que par rapport à l'une des deux variables (la variable de droite ou celle de gauche respectivement), afin de prendre ses valeurs dans $\mathrm{D}(X)$. Le problème essentiel est qu'il n'y a pas de structure d'algèbre naturelle quand on dérive par rapport à une seule des variables, et par suite le théorème PBW inverse est inapplicable. Il faut donc construire la structure d'algèbre à la main, ce qui est une difficulté technique importante. Enfin, le troisième point a été démontré initialement dans \cite{yu_todd_2015} par des méthodes entièrement différentes.

\begin{corollary}
Si $(X, \sigma)$ est un cycle quantifié modéré dans $Y$, alors l'algèbre $\mathrm{Ext}^*_{\mathcal{O}_Y}(\mathcal{O}_X, \mathcal{O}_X)$ ne dépend que du second voisinage formel de $X$ dans $Y$.
\end{corollary}

\section{Conclusion et perspectives}
Si $X_1$ et $X_2$ sont deux sous-variétés algébriques lisses fermées dans une variété algébrique lisse ambiante $Y$, les algèbres $\mathscr{T}or^*_{\mathcal{O}_Y}(\mathcal{O}_{X_1}, \mathcal{O}_{X_2})$ et $\mathscr{E}xt^*_{\mathcal{O}_Y}(\mathcal{O}_{X_1}, \mathcal{O}_{X_2})$ sont intimement reliées à la géométrie de l'intersection de $X_1$ et $X_2$ dans $Y$, l'exemple le plus frappant étant la 
formule des Tor de Serre \cite{serre_local_2000}. Certaines hypothèses globales géométriques fournissent des structures additionnelles
algébriques sur les deux algèbres précitées: par exemple si $Y$ est une variété hyperkählerienne et $X_1$, $X_2$ sont lagrangiennes, il est démontré dans \cite{BF} que la première algèbre est naturellement une algèbre de Gerstenhaber, et la seconde une algèbre de Batalin-Vilkovisky. 
\par \medskip
Dans le cas $X_1=X_2=Y$, comprendre les algèbres $\mathcal{O}_X \otimes_{\mathcal{O}_Y}^{\mathbb{L}} \mathcal{O}_X$ et $\mathcal{RH}om_{\mathcal{O}_Y}(\mathcal{O}_X, \mathcal{O}_X)$ est un problème compliqué, certains éléments de réponse étant fournis dans les articles \cite{arinkin_when_2012}, \cite{calaque_lie_2014} et \cite{grivaux_derived_2015} et \cite{calaque_ext_2017}:

\begin{enumerate}
\item[--] L'article \cite{calaque_pbw_2013} donne une réponse en toute généralité, mais la structure d'algébroïde de Lie sur $\mathrm{N}_{X/Y}[-1]$ n'est pas facile à manipuler en pratique.
\item[--] L'article \cite{arinkin_when_2012} fournit un isomorphisme multiplicatif \[ \mathcal{O}_X \otimes_{\mathcal{O}_Y}^{\mathbb{L}} \mathcal{O}_X \simeq \mathrm{S}(\mathrm{N}^*_{X/Y}[1])\] ainsi qu'un isomorphisme additif \[\mathcal{RH}om_{\mathcal{O}_Y}(\mathcal{O}_X, \mathcal{O}_X) \simeq  \mathrm{S}(\mathrm{N}_{X/Y}[-1])\] pour les cycles quantifiés.
\item[--] L'article \cite{calaque_ext_2017} dont il est question dans ce texte permet de décrire explicitement $\mathcal{RH}om_{\mathcal{O}_Y}(\mathcal{O}_X, \mathcal{O}_X)$ pour les cycles quantifiés modérés.
\end{enumerate}
Pour aller plus loin, on peut poursuivre dans différentes directions.
\begin{enumerate}
\item[--] Parmi les cycles quantifiés, on peut chercher à étudier d'autres exemples qui ne sont pas en général modérés, par exemple les lieux fixes d'une action par un groupe réductif. Dans le cas de $\mathbb{Z}/{2\mathbb{Z}}$, cela correspond à étudier une version géométrique des paires symétriques.
\item[--] Si on s'affranchit de l'hypothèse de quantification, la situation devient bien plus compliquée: il faut étudier en détail la structure d'algébroïde de Lie à homotopie près sur $\mathrm{N}_{X/Y}[-1]$ construite dans \cite{calaque_lie_2014}.
\end{enumerate}

\bibliographystyle{plain}
\bibliography{smf.bib}

\end{document}